\documentclass[11pt,reqno]{amsart}
\usepackage{amsmath,amsthm,amsfonts,amssymb,amscd,overpic}
\usepackage{graphicx,overpic}
\numberwithin{equation}{section}

\newcommand{\eqdef}{\stackrel{\scriptscriptstyle\rm def}{=}}

\def\ZZ {{\mathbb Z}}\def\NN {{\mathbb N}}\def\RR {{\mathbb R}}\def\CC {{\mathbf C}}

\def\Si{\Sigma}\def\La{\Lambda}\def\Ga{\Gamma}

\def\de{\delta}

   \def\cM{{\mathcal M}}              \def\cE{{\mathcal E}}

\DeclareMathOperator{\card}{card}

\def\diff1{\mbox{{\rm Diff}}^1(M)}
\def\diff2{\mbox{{\rm Diff}}^2(M)}

\def\stam{{\mu^\star}}
\def\stam1{{\mu^\star_1}}

\def\st{s}
\def\sst{{ss}}

\def\ut{u}
\def\uut{{uu}}

\def\loc{{\operatorname{loc}}}

\newtheorem{theo}{Theorem}

\newtheorem{theor}{Theorem}[section]
\newtheorem{lemm}[theor]{Lemma}
\newtheorem{step}[theor]{Step}
\newtheorem{claim}[theor]{Claim}
\newtheorem{coro}[theor]{Corollary}

\newtheorem{defi}[theor]{Definition}
\newtheorem{prop}[theor]{Proposition}

\newtheorem{rema}[theor]{Remark}

\addtocounter{figure}{0}

\title[Rich phase transitions]{Rich phase transitions in step skew-products}

\author[L.~J.~D\'iaz]{Lorenzo J. D\'\i az}
\address{Departamento de Matem\'atica PUC-Rio, Marqu\^es de S\~ao Vicente 225, G\'avea, Rio de Janeiro 225453-900, Brazil}
\email{lodiaz@mat.puc-rio.br}

\author[K.~Gelfert]{Katrin~Gelfert}
\address{Instituto de Matem\'atica UFRJ, Av. Athos da Silveira Ramos 149, Cidade Universit\'aria - Ilha do Fund\~ao, Rio de Janeiro 21945-909,  Brazil}\email{gelfert@im.ufrj.br}

\author[M.~Rams]{{Micha\l} Rams} \address{Institute of Mathematics, Polish Academy of Sciences, ul. \'{S}niadeckich 8, 00-956 Warszawa, Poland}
\email{m.rams@impan.gov.pl}

\thanks{This paper was partially supported by CNPq, Faperj, and Pronex (Brazil),  CODY (EU) as well as the MNiSWN grant N201 607640 (Poland). The authors thank the hospitality of IM\,PAN and IM\,UFRJ. The authors thank S.~Crovisier for helpful discussions.}

\begin{document}
\maketitle

\begin{abstract}
	We present examples of partially hyperbolic and topologically transitive local diffeomorphisms defined as skew products over a horseshoe which exhibit rich phase transitions for the topological pressure. This phase transition follows from a gap in the spectrum of the central Lyapunov exponents. It is associated to the coexistence of two equilibrium states with positive entropy.
	The diffeomorphisms mix hyperbolic behavior of different types. However, in some sense the expanding behavior is not dominating which is indicated by the existence of a measure of maximal entropy with nonpositive central exponent.
\end{abstract}

\section{Introduction}

Given a compact metric space $\Lambda$ and a continuous map $f\colon \Lambda\to\Lambda$, the arising dynamical system can be studied from various points of view. On the one hand, one can investigate the topological dynamics determined by $f$. On the other hand, one can investigate the set $\cM$ of $f$-invariant Borel probability measures and study measure theoretic aspects of the dynamics. Both sides are closely linked with each other. This link is characterized by the thermodynamic formalism and specified by the topological pressure functional (see~\cite{Wal:81} for full details). Given a continuous function $\varphi\colon\Lambda\to\RR$, its topological pressure $P(\varphi)$ is defined in purely topological terms. It can be expressed in measure-theoretic terms via the variational principle
\begin{equation}\label{e.introvp}
	P(\varphi)=\sup_{\mu\in\cM}\Big( h_\mu(f)+\int\varphi\,d\mu\Big),
\end{equation}
where $h_\mu(f)$ denotes the entropy of the measure $\mu$.

The set $\cM$ equipped with the weak topology forms a compact convex space that can have an extremely complicated structure. It is natural to aim for a closer understanding of this structure. One way is to characterize measures that are ``relevant'' and ``designated" in a certain sense,
for example, to focus on equilibrium measures. An invariant measure $\nu$ is said to be an \emph{equilibrium measure} or \emph{equilibrium state} of $\varphi$ with respect to $f$ if it attains the supremum in the variational principle~\eqref{e.introvp}
(the measure $\nu$ maximizes what is sometimes also called the \emph{free energy} of the \emph{potential} $\varphi$).
Note that an equilibrium state $\nu$ for the zero potential $\varphi=0$ is a measure of maximal entropy. 

To show existence and uniqueness and to establish further specific properties of equilibrium measures are among the main problems in the  thermodynamic formalism. On the other hand, particularly interesting are examples where existence or uniqueness fails. In many cases the coexistence of equilibrium states for some given potential is closely related to so-called phase transitions.
Following nowadays standard notation, we say that the pressure function $t\mapsto P(t\varphi)$, $t\in\RR$, exhibits a \emph{phase transition} at a characteristic parameter $t_c$ if it fails to be real analytic at $t_c$. We say that it has a \emph{first order phase transition} at $t_c$ if it fails to be differentiable at $t_c$.

One major line of research in the thermodynamic formalism considers abstract dynamical systems such as Markov shifts and potentials with a certain regularity. See the classical texts by Ruelle~\cite{Rue:04} and Bowen~\cite{Bow:08} as well as the collection by Sarig~\cite{Sar:09}.
However, in the present paper we focus on \emph{smooth} dynamical systems.

We will follow a classical approach to analyze smooth systems by studying their ``basic pieces''. In the realm of uniformly hyperbolic dynamics such pieces are formed by the \emph{basic sets} (sets that are compact, invariant, uniformly hyperbolic, topologically transitive, and locally maximal). Beyond uniformly hyperbolic dynamics a natural line of generalization is the investigation of \emph{homoclinic classes} (such sets are topologically transitive and contain a dense subset of hyperbolic periodic points, see Definition~\ref{def:homcla}).

Loosely speaking, when a system dynamically splits into basic pieces then this should also be reflected by the structure of its ``dual" $\cM$.
Certainly, if a system has several transitive components (though they could be intermingled) then this will be reflected dynamically.
Dobbs~\cite{Dob:09}, for example, explains the mechanism that gives rise to phase transitions related to non-transitive behavior in renormalizable unimodal maps.
If, however, a dynamical system is topologically transitive but there exist pieces that are ``exposed'' in a sense that dynamically and topologically they form extreme points  then we still can observe the phenomenon of phase transitions and coexistence of equilibrium states. To further support this point of view we will  discuss some examples and explain what we mean by ``exposed".

One well-understood case is when the transitive dynamics ``splits'' into a hyperbolic piece and a nonhyperbolic piece. Let us recall those examples of interval maps $f$ and the classically considered potential $\varphi(x)=-\log\,\lvert f'(x)\rvert$. In the example of Manneville and Pomeau~\cite{ManPom:80} the interval $\Lambda=[0,1]$ splits into the single parabolic fixed point $x=0$ and the remaining set $(0,1]$. The pressure $P(t\varphi)$ exhibits a phase transition at $t_c=1$ that is related to a coexistence of the Dirac measure $\delta_0$ and an acip (absolutely continuous invariant probability measure) as equilibrium states.
In the case of the Chebyshev polynomial $x\mapsto 4x(1-x)$ the interval $\Lambda=[0,1]$ is transitive and exposes the post-critical fixed point $x=0$ (though the potential fails to be continuous due to the singularity) and exhibits a phase transition at $t_c=-1$.
An example of Bruin~\cite{Bru:03} discusses quadratic maps with (several) ergodic measures supported on minimal Cantor sets with zero Lyapunov exponents that are coexisting with an acip.
Note that in all these examples the nonhyperbolic part has zero entropy.
See~\cite[Section 7]{BruTod:09},~\cite{IomTod:}, and~\cite{CorRiv:10} for further discussion.

Of similar spirit are the examples of the Julia set $\Lambda=J$ of a polynomial or rational exceptional map $f$ on the complex plane discussed by Makarov and Smirnov in~\cite{MakSmi:96,MakSmi:00}. Here, the Julia set possesses periodic points that are ``dynamically exposed'' in the sense that they are immediately post-critical (there is no branch of preimages dense in $J$ and disjoint with critical points, see also~\cite{GelPrzRamRiv:}). Chebyshev polynomials of degree $d\ge2$ are particular examples. In each of these cases the post-critical set is finite and thus carries only measures with zero entropy.
Those measures are equilibrium states  associated to a phase transition of the pressure function $t\mapsto P(-t\varphi)$ at a characteristic parameter. 

In the present paper we present examples of local diffeomorphisms $F$ with a locally maximal nonhyperbolic set  $\Lambda\subset\RR^3$ that is at the same time a homoclinic class.
We do not aim for generality but instead try to provide the simplest example possible. For that we choose a map that is a skew-product of interval diffeomorphisms over a Smale horseshoe with three legs.
Although the dynamics on $\Lambda$ is topologically transitive, the spectrum of the Lyapunov exponents associated to the one-dimensional central direction $E^c$ contains positive and negative values and has a gap.
In this example the potential $t\varphi(x)=-t\log\,\lVert dF|_{E_c}\rVert$ is continuous and equilibrium states for $t\varphi$ exist for every $t\in\RR$ (see~\cite{DiaFis:11}).
The spectral gap is immediately related to a phase transition of the pressure function $t\mapsto P(t\varphi)$ at some characteristic parameter $t_c$. Moreover, for $t_c\varphi$ there exist two equilibrium states both of positive entropy (we call this a \emph{rich phase transition}).
To our knowledge, this is the first example of a phase transition in a topologically transitive local diffeomorphism associated to several equilibrium states with positive entropy.
This work is an extension of~\cite{DiaGel:,LepOliRio:10} where topological properties and phase transitions related to homoclinic classes are studied. In these examples the exposed sets are single fixed points.

Inside the locally maximal set $\Lambda$ coexist intermingled sets of different type of hyperbolicity which can be seen, for example, from the coexistence of periodic points with unstable manifolds of different dimensions. However, there exists a measure of maximal entropy with nonpositive central Lyapunov exponent. Hence, we may conclude that the dynamics on $\Lambda$ is not predominantly expanding.
Moreover, in some parameter range this measure is unique and its central exponent is negative.

Besides these dynamical features, our examples possess also topologically a rich structure of the fibers of the skew-product map. Following the approach in~\cite{DiaGel:}, one can show that there exists uncountably many fibers that contain a single point only and uncountably many fibers that contain a continuum.

Let us briefly explain to what corresponds an ``exposed piece of dynamics" in our example. We consider genuinely nonhyperbolic homoclinic classes $\Lambda$  containing infinitely many hyperbolic periodic points of different type. Although this class is transitive, it properly contains a ``lateral" horseshoe $\Lambda_{02}\subset\Lambda$ whose saddles are not homoclinically related to periodic hyperbolic points outside $\Lambda_{02}$. This lateral horseshoe is a kind of extreme of $\Lambda$. One of the equilibrium states involved in the phase transition is supported on the horseshoe $\Lambda_{02}$ that has positive central Lyapunov exponents and positive entropy. The other coexisting state lives on $\Lambda\setminus\Lambda_{02}$ and also has positive entropy.

Let us conclude with a heuristic remark. That a homoclinic classes is properly contained in a bigger one seems to be the underlying mechanism for the gap in the spectrum and hence for the phase transition.  However, this configuration is somehow atypical. Indeed, for typical $C^1$ diffeomorphisms homoclinic classes either are disjoint or coincide~\cite{CarMorPac:03} and the spectrum of Lyapunov exponents has no gaps~\cite{AbdBonCroDiaWen:08}.

The paper is organized as follows. We provide the details of our example in  Section~\ref{s:2}. In Section~\ref{s:3} we explain its topological properties and establish topological transitivity, we postpone the proof to Section~\ref{s:7}.  In Section~\ref{s:4} we prove the existence of a gap in the spectrum of central Lyapunov exponents, see Proposition~\ref{p.beta} and Corollary~\ref{c.spectrum}.
In Section~\ref{s:5} we prove the existence of a rich phase transition (see Theorem~\ref{t.hund}). Moreover, in Section~\ref{s.52} we discuss periodic fibers that give rise to a measure of maximal entropy with nonpositive or even negative central exponent (see Proposition~\ref{p.maxent1} and Corollary~\ref{c.entuni}). In Section~\ref{s:6} we briefly discuss more general potentials and provide a sufficient condition for the existence of a phase transition in our setting.

\section{A class of skew-products}\label{s:2}

In this section we construct the class of maps that will be studied in this paper.
Consider the cube $\widehat \CC=[0,1]^2$ and a diffeomorphism $\Phi$ defined on $\RR^2$
having a horseshoe $\Ga$ in $\widehat\CC$ conjugate to the
full shift $\sigma$ of three symbols.
Denote by  $\varpi\colon \Ga \to \Si_3$ the conjugation map
$\varpi \circ \Phi=\sigma \circ \varpi$. We consider the following naturally associated sub-cubes
$\widehat \CC_0$, $\widehat \CC_1$, $\widehat \CC_2$ of $\widehat \CC$
\[
	\widehat \CC_i\eqdef
	\big\{X\in\CC\colon \varpi (X) = (\ldots\xi_{-1}.\xi_0\xi_1\ldots)
	\text{ with }\xi_0=i\big\}.
\]	
Let $\CC= \widehat \CC\times [0,1]$ and $\CC_i= \widehat \CC_i\times [0,1]$. We consider the map $F\colon\CC\to \RR^3$ defined by
\begin{equation}\label{e.defF}
    F(\widehat x,x) \eqdef
        (\Phi(\widehat x), f_i(x))
    \quad\mbox{ if }\,\,
    X=(\widehat x,x)\in \widehat \CC_i\times [0,1],
\end{equation}
where $f_i \colon [0,1]\to [0,1]$, $i=0$, $1$, $2$, are assumed to be $C^1$
injective interval maps satisfying properties that we are going to specify now.
\begin{figure}
\begin{minipage}[c]{\linewidth}
\centering
\begin{overpic}[scale=.55
  ]{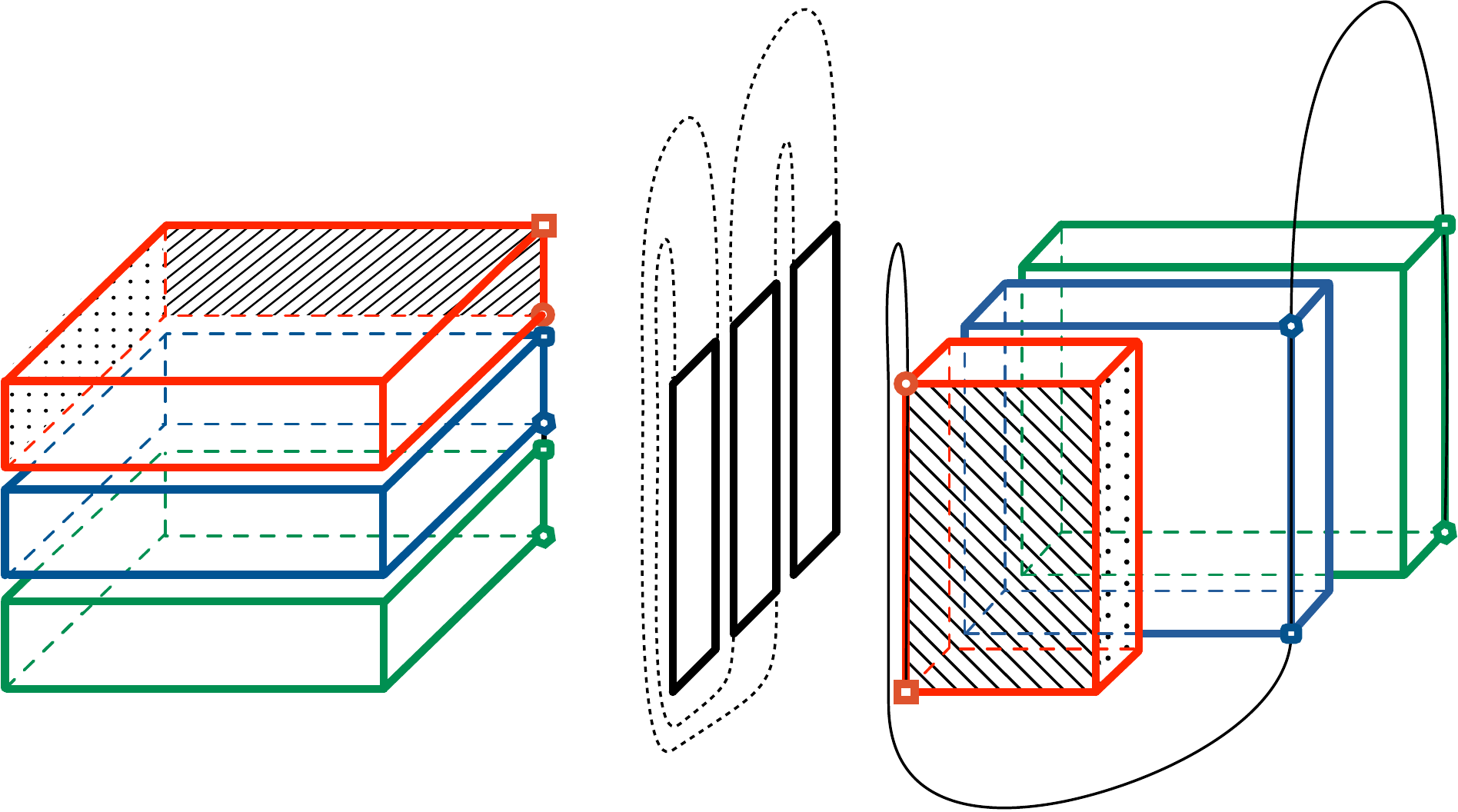}
 \end{overpic}
\caption{Construction of the maximal invariant set $\Lambda$}
\label{fi.neu}
\end{minipage}
\end{figure}
To produce a simple example, we will assume that $\Phi|_{\CC_i}$ is affine.
We also assume that
 the rate of expansion (contraction) of the
horseshoe is stronger than any expansion (contraction) of $f_0$, $f_1$, and $f_2$.
In this way the $DF$-invariant  splitting $E^\sst\oplus E^c\oplus E^\uut$
given by
\begin{equation}\label{e.splitt}
    E^\sst\eqdef\RR^s\times\{(0^u,0)\}, \,
    E^c\eqdef\{(0^s,0^u)\}\times\RR, \,
    E^\uut\eqdef\{0^s\}\times\RR^u\times\{0\}
\end{equation}
is dominated and $E^{ss}$ and $E^{uu}$ are uniformly hyperbolic.
We denote by $W^{ss}$ and $W^{uu}$ the corresponding strong stable and strong unstable manifolds associated  to $E^{ss}$ and $E^{uu}$.
\smallskip

The following conditions (F0),(F1), and (F2) will imply that the system of the fiber maps $\{f_0,f_1,f_2\}$ is of cycle type and mixes expansion and contraction behavior -- compare also Figure~\ref{fi.neu}.

\begin{itemize}
\item[\textbf{(F0)}] 
    The map $f_0$ is increasing and has exactly two hyperbolic fixed
    points, the point $q_0=0$ (repelling) and the point $p_0=1$ (attracting).
    Let $\beta_0=f_0^\prime (0)>1$ and $\lambda_0=f_0^\prime (1)\in
    (0,1)$. Moreover, $\lambda_0\le f_0^\prime(x)$ for all $x \in [0,1]$.
    \\[-0.3cm]
\item[\textbf{(F1)}]
    The map $f_1$ is an affine contraction with negative
    derivative
    \[
    f_1(x)\eqdef \gamma\,(1-x),
    \]
    where $\gamma\ge\lambda_0$. We denote by $p_1$ the attracting fixed point of $f_1$. Note that $f_1(1)=0$ (cycle condition).
    \\[-0.3cm]
\item[\textbf{(F2)}]
    The map $f_2$ is increasing and has two hyperbolic fixed
    points, the point $q_2=0$ (repelling) and the point $p_2\in(0,1)$ (attracting).
    We have $\beta_2=f_2^\prime (0)>1$.
    \\[-0.3cm]
\end{itemize}	

The next condition (F01) guarantees that suitable compositions of $f_0,f_1$ exhibit some expanding behavior (see Step~\ref{s.expanding}).
We formulate this condition in the simplest case where
the derivative $f_0'$ is decreasing in $[0,1]$.

\smallskip
\begin{itemize}
\item[\textbf{(F01)}]
The derivative $f_0'$ is decreasing in $[0,1]$ and satisfies
$$
\gamma
\, \left(
	\frac{\lambda_0^3 \, (1-\lambda_0)}{1-\beta_0^{-1}} \right)>1.
$$
\end{itemize}
Note that given $\gamma, \lambda_0\in(0,1)$, this condition
is clearly satisfied if $\beta_0>1$ is sufficiently close to $1$.

To establish the existence of phase transitions, we will require one further property (F012) giving constraints to the variation of $f_0'$ and $f_2'$ and to $f_0'(1)$. It  will be specified in Section~\ref{s:4}.

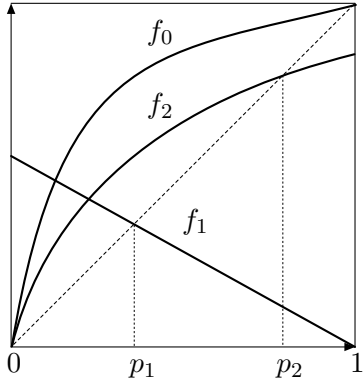
\begin{figure}
\begin{minipage}[c]{\linewidth}
\centering
\begin{overpic}[scale=.40
  ]{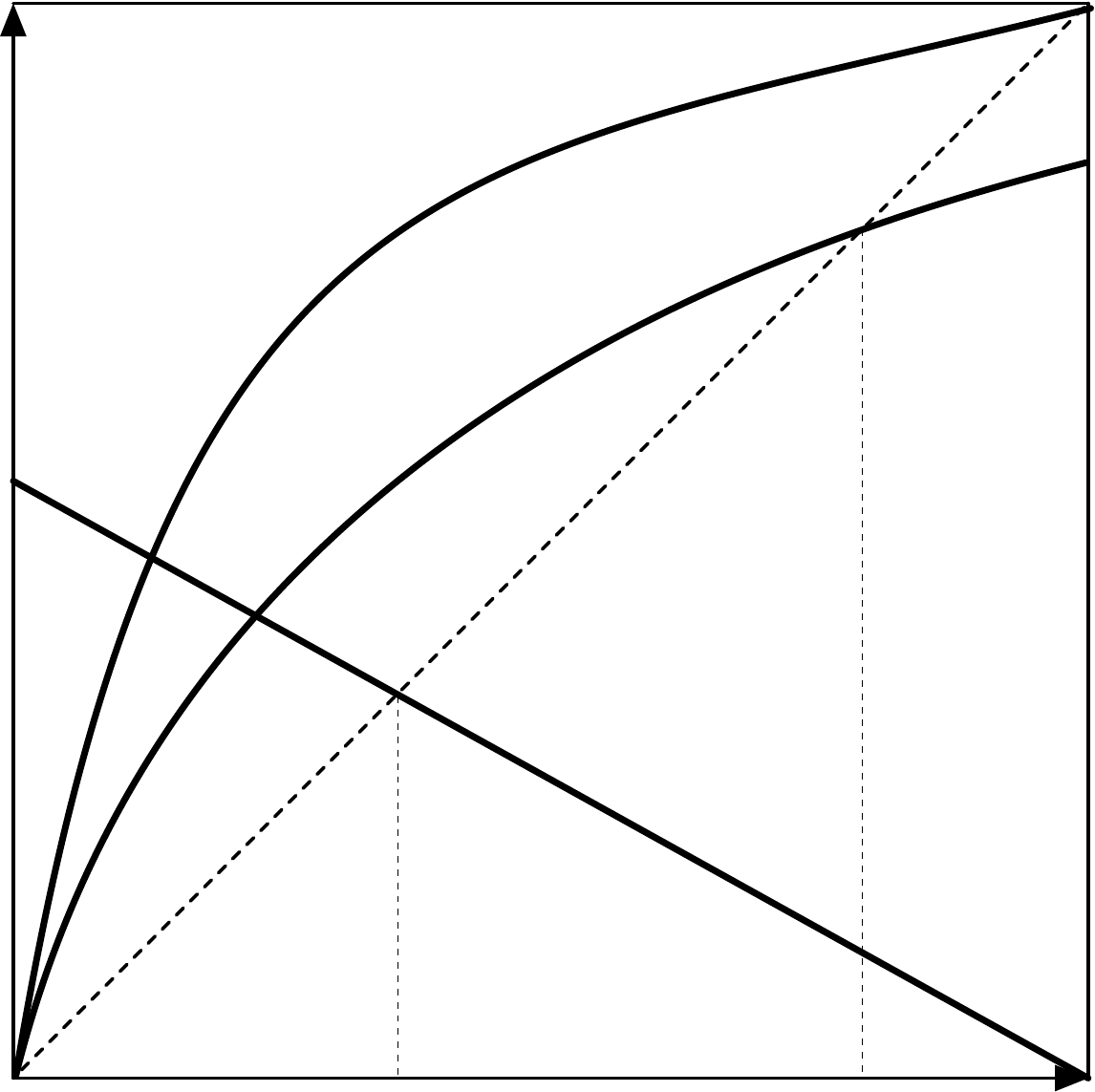}
      	\put(40,89){$f_0$}	
	\put(50,35){$f_1$}	
        \put(40,68){$f_2$}	
	\put(0,-6){$0$}	
	\put(98,-6){$1$}	
      	\put(35,-6){$p_1$}	
      	\put(77,-6){$p_2$}	
      	\end{overpic}
\end{minipage}
\smallskip
\caption{The maps of the IFS (in a special case $\beta_2<\beta_0$)}
\end{figure}

\section{Dynamical properties}\label{s:3}

We  introduce some notations and state some dynamical properties of $F$.
The skew product structure of $F$ allows us to reduce the study of its dynamics to the study of the IFS associated by the maps $f_i$ (see also Section~\ref{s:4}).

Consider the sequence space $\Sigma_3=\{0,1,2\}^\ZZ$ endowed with the metric $d(\xi,\eta)=\sum_{i\in\ZZ}2^{-\lvert i\rvert}\lvert\xi_i-\eta_i\rvert_\ast$ for $\xi=(\ldots\xi_{-1}.\xi_0\xi_1\ldots)$, $\eta=(\ldots\eta_{-1}.\eta_0\eta_1\ldots)\in\Sigma_3$, where $\lvert\xi_i-\eta_i\rvert_\ast=1$ if $\xi_i\ne\eta_i$ and $0$ otherwise.
Every sequence $\xi\in\Sigma_3$ is given by $\xi=\xi^-.\xi^+$, where $\xi^+\in\Sigma_3^+\eqdef\{0,1\}^{\NN_0}$ and $\xi^-\in\Sigma_3^-\eqdef\{0,1\}^{-\NN}$.
We denote by $(\xi_0\ldots\xi_{m-1})^\ZZ$ the periodic sequence of period $m$ such that $\xi_i=\xi_{i+m}$ for all $i$ and always refer to the least period of a sequence.

\subsection{Fixed points and invariant sets}

Denote by $\theta_i=(\theta_i^s,\theta_i^u)=\varpi^{-1}(i^\ZZ)$, $i=0$, $1$, $2$, the fixed points of the horseshoe map $\Phi$.
The structure of the horseshoe and the choice of $f_0,f_1,f_2$ imply that $F$ possesses five fixed points given by
\begin{equation}\label{e.defPQ}\begin{split}
    &P_0\eqdef(\theta_0,1),\quad
    P_1\eqdef(\theta_1,p_1),\quad
    P_2\eqdef(\theta_2,p_2),\\
    &Q_0\eqdef(\theta_0,0),\quad
    Q_2\eqdef(\theta_2,0).
\end{split}\end{equation}
Observe that the {\emph{u-index}} (dimension of the unstable manifold) of $P_i$ is $1$ while the one of $Q_i$ is $2$, for all $i$.

Let us consider the lateral two-legged horseshoe of $F$
\begin{equation}\label{e.twoleg}
    \Lambda_{02}\eqdef \Gamma_{02}\times \{0\}, \quad\text{where}\quad
    \Gamma_{02}\eqdef  \varpi^{-1}(\{0,2\}^\ZZ).
\end{equation}
Notice that $\Lambda_{02}$ is invariant with respect to $F$ since for every $\theta\in\varpi^{-1}(\{0,2\}^\ZZ)$ and $i=0$, $2$ we have $F(\theta,0)=(\Phi(\theta),f_i(0))=(\theta',0)$ with $\theta'\in\varpi^{-1}(\{0,2\}^\ZZ)$.
This lateral horseshoe is topologically transitive, uniformly hyperbolic, and contains the saddles $Q_0$ and $Q_2$.

\subsection{Homoclinic classes}

We will focus on the maximal invariant set $\Lambda$ of $F$ in the cube $\CC=[0,1]^3$, $\Lambda \eqdef \bigcap_{i\in \mathbb{Z}} F^i (\CC)$, that will be a special type of transitive set called  a homoclinic class. Inside the set $\Lambda$ coexist intermingled hyperbolic sets of different ``types'' ($u$-indices). This will give rise to the existence of heterodimensional cycles associated to periodic points in $\Lambda$. This is the underlying mechanism to produce a rich dynamics mixing hyperbolicity of different types.

\begin{defi}[Homoclinic class]\label{def:homcla}
    {\rm
    Given a saddle point $P$ of $F$ its  \emph{homoclinic class} is the closure of the transverse intersections of the stable and unstable manifolds of the orbit of $P$. Two saddles $P$ and $Q$ are \emph{homoclinically related} if the invariant manifolds of their orbits meet cyclically and transversely. A homoclinic class is \emph{non-trivial} if it contains at least two different orbits.

    In the following we restrict our attention to the dynamics inside the cube $\CC$. We call the closure of the set of points that are in the transverse intersections of the stable and unstable manifolds of the orbit of $P$ and
whose  orbit is entirely contained in $\CC$ the \emph{homoclinic class of $P$ relative to $\CC$}. We denote this set by $H(P,F)$.
    }
\end{defi}

\begin{rema}\label{r.funda}{\rm
	Observe first  that homoclinically related saddles all have the same $u$-index.
Note also
that the (relative) homoclinic class $H(P,F)$ coincides with the closure of all saddle points that are homoclinically related to $P$ relative to the cube $\CC$ (i.e.,
the orbits of the transverse intersections are contained in $\CC$).
  However, this
closure may contain periodic points that are not homoclinically related to $P$. Indeed, this paper illustrates such a situation and, in fact, is an essential ingredient of our example. Finally, a homoclinic class is always
{\emph{transitive}} (existence of a dense orbit) and uncountable if non-trivial.
}\end{rema}

Finally, we see that $\Lambda =\cap_{i\in \ZZ} F^i(\CC)$ is nonhyperbolic and  transitive.

\begin{prop}\label{p.transitive}
	There is a saddle $Q^*$ of $u$-index two such that the set $\Lambda$ is the homoclinic class of $Q^\ast$ relative $\CC$. In particular, the set $\Lambda$ is topologically transitive.
\end{prop}

This proposition is a version of the results in~\cite{DiaGel:} considering skew product dynamics over the shift of two symbols whose central dynamics is given by the maps $f_0$ and $f_1$. The only difference is that the structure of $\Lambda$ here is slightly more complicated due to the existence of the additional  ``leg'' of the horseshoe and its associated map $f_2$. This extra ``leg'' is also responsible for the existence of the lateral horseshoe and gives rise to more combinatorics in the orbits. We postpone the proof to Section~\ref{s:7}.

From similar observations we conclude also the following relations showing the special nature of the lateral horseshoe
\[
	\Lambda_{02}=H(Q_0,F)=H(Q_2,F)\subsetneq H(P_0,F)
	\subset\Lambda.
\]

\section{Lyapunov exponents of the IFS. Spectral gap}\label{s:4}

In this section we study dynamical properties of the underlying  iterated function system (IFS) generated by the maps $f_i$, establishing the existence of a gap in the spectrum of central Lyapunov exponents (Proposition~\ref{p.beta}).
This gap will correspond to a gap in the spectrum of the central exponents of diffeomorphism $F$ due to the skew product structure (see Corollary~\ref{c.spectrum}).

We use the following notation for concatenated maps of the IFS.
Given a \emph{finite} sequence
$(\xi_0\ldots \xi_m)$, $\xi_i\in\{0,1,2\}$, let
\[
    f_{[\xi_0\ldots\,\xi_m]}
    \eqdef f_{\xi_m} \circ \cdots \circ f_{\xi_1}\circ f_{\xi_0} \colon [0,1]\to [0,1].
\]
Given a finite sequence $(\xi_{-m}\ldots\xi_{-1})$ let
\[
    f_{[\xi_{-m}\ldots\,\xi_{-1}.]}
    \eqdef  (f_{\xi_{-1}}\circ\ldots\circ f_{\xi_{-m}})^{-1}.
\]
Similarly, given a finite sequence $(\xi_{-m}\ldots\xi_{-1}.\xi_0\ldots \xi_n)$, let
\[
    f_{[\xi_{-m}\ldots\,\xi_{-1}.\xi_0\ldots\,\xi_n]}
    \eqdef  f_{[\xi_0\ldots\,\xi_n]} \circ f_{[\xi_{-m}\ldots\,\xi_{-1}.]}.
\]
An one-sided infinite sequence $(\ldots\xi_{-2}\xi_{-1}.)\in\Sigma_3^-$ is said to be \emph{admissible} for a point $x$ if the map $f_{[\xi_{-m}\ldots\,\xi_{-1}.]}$ is well-defined at $x$ for all $m\ge 1$.
By writing $(x,\xi)$ we always mean that $\xi$ is admissible for $x$.

Given $p\in[0,1]$ and a sequence $\xi=(\ldots\xi_{-1}.\xi_0\xi_1\ldots)\in\Sigma_3$ that is admissible for $p$, the \emph{(forward) Lyapunov exponent} of $p$ with respect to $\xi$ is defined by
\[
    \chi(p,\xi)\eqdef
    \lim_{n\to\infty}\frac{1}{n}
    \log \, \big\lvert (f_{[\xi_0\ldots\,\xi_{n-1}]})'(p)\big\rvert
\]
whenever this limit exists.
Otherwise we denote by $\underline\chi(p,\xi)$ and $\overline\chi(p,\xi)$ the \emph{lower} and the \emph{upper Lyapunov exponent} defined by taking the lower and the upper limit, respectively.
Note that, in fact, $\chi(p,\xi)$ depends only on the positive part $\xi^+$ of $\xi=(\xi^-.\xi^+)$ only. Hence, when considering exponents of a pair $(p,\xi)$ in the following, we will disregard the hypothesis that $\xi$ is admissible.

Let us consider the symbolic description of the lateral horseshoe $\Lambda_{02}$ together with its stable manifold. This is given by the set $\cE$ of ``exceptional points'' of the IFS:
\[\begin{split}
    \cE\eqdef &
    \Big\{ (0,\xi)\colon
    		\xi\in \{0,2\}^\ZZ \Big\}\,\cup \\
		&
    \Big\{\big(1,(0^{-\NN}.0^k\,1\,\xi^+)\big),\,
    		\big(0,(0^{-\NN} 1\,\xi_k.\xi^+)\big) \colon	
    		k\ge0, \xi_k\in \{0,2\}^k, \xi^+\in\{0,2\}^\NN\Big\}.
\end{split}\]
As $\cE$ codes all points in the lateral horseshoe $\Lambda_{02}$ together with its stable manifold, its central Lyapunov spectrum is an interval. 

Let
    	\[
    		\beta_{02}^-\eqdef\min\{\beta_0,\beta_2\},\quad
		\beta_{02}^+\eqdef\max\{\beta_0,\beta_2\}.
    	\]
To show the following lemma, it is enough  to observe that, by construction,
$\beta^-_{02}$ and $\beta^+_{02}$ are the smallest and largest central exponents of $\Lambda_{02}$.

\begin{lemm}
	$\displaystyle \big\{\overline\chi(p,\xi)\colon (p,\xi)\in \cE \,\big\}
                = [\log\beta^-_{02},\log\beta^+_{02}\,].$
\end{lemm}

To prove the main result of this section, we need an additional assumption to be satisfied.
\smallskip
\begin{itemize}
\item[\textbf{(F012)}]
	We have $f_0^\prime(x), f_2'(x)\le\beta^+_{02}$ for all $x \in [0,1]$.
	There exists an interval $H=[0,\delta]$ such that
	$H$ and
	$H'=f_1^{-1}(H)$
	satisfy
	\begin{equation}\label{e.H0}\begin{split}
	f_1(H)\cap H =\emptyset,
	\,\,\, f_1&(H')\cap H' =\emptyset,
	\,\,\, f_1([0,1])\cap H' =\emptyset,\\
	&f_2([0,1])\cap H'=\emptyset.
\end{split}\end{equation}
    Assume that
	\[	
	\beta' \eqdef \max\{f_0'(x), f_2'(x)\colon x\notin H\}
		<\beta^-_{02}.
	\]
	Moreover, let    	
    \[\begin{split}
	\beta_{H}&\eqdef  \min\{f_0'(x),f_2'(x) \, \colon \, x\in H\}
		<\beta^-_{02},\\
	\lambda' &\eqdef \max\{f_0'(x) \, \colon \, x\in H'\}
		<1
	\end{split}\]
	and assume that
   \begin{equation}\label{e.derivatives}
    	\lvert\log\lambda_0\rvert\frac{\log\beta^+_{02}}{\log\beta_H}
	-\lvert\log\lambda'\rvert
	+\frac 2 3\log\beta^+_{02}<\frac 3 4  \log\beta'.
    \end{equation}
    Finally, let us also assume that with $L\ge1$ satisfying
    \begin{equation}\label{edefL}
    	L> 4\,\frac{\lvert\log\lambda_0\rvert\,\log\beta^+_{02}}{\log\beta_H\,\log\beta'}
    \end{equation}
    we have
    \begin{equation}\label{e.derivatives2}
    	\frac{\log((\beta^+_{02})^L\,\gamma)}{L+1} < \log\beta'.
 	 \end{equation}
\end{itemize}
\smallskip
Clearly, (F012) can be guaranteed if $\beta_2$ is close enough to $\beta_0$ and if $f_2$ is non-linear close to $0$ and does not contract the length of the unit interval too much. Then we can choose $\delta$ very small and can guarantee that $\beta_H$ is close to $\beta^+_{02}$ and that $\lambda'$ is close to $\lambda_0$ in order to guarantee~\eqref{e.derivatives}.
\smallskip

We obtain the following gap of the full range of possible exponents.

\begin{prop}\label{p.beta}
    Let
    \[
                \widetilde\chi\eqdef
                \sup\big\{\overline\chi(p,\xi)\colon
                		(p,\xi)\notin \cE \big\}.
    \]
    Under the hypothesis (F012) we have $\widetilde\beta\eqdef \exp\widetilde\chi\in(1,\beta^-_{02})$.
\end{prop}

\begin{proof}
Arguing by contradiction, let us assume that for every $\varepsilon>0$ there exist a point $p\in[0,1]$ together with a sequence $\xi=(\ldots\xi_{-1}.\xi_0\xi_1\ldots)\in\Sigma_3$ such that $(p,\xi)\notin\cE$ and that $\overline\chi(p,\xi)>\log\beta^-_{02}-\varepsilon$.

The Lyapunov exponent of a trajectory is the average of $\log\,\lvert f_i'\rvert$ along the trajectory. Hence, if the upper Lyapunov exponent of a trajectory is greater than $\log\beta^-_{02}-\varepsilon$, the trajectory must return to $H$ infinitely many times.
The only way to enter $H$ is by coming from $H'$ after applying $f_1$, and the only way to get into $H'$ is by applying $f_0$.
To reach a contradiction, we only need to prove that the average of $\log |f_i'|$ along the piece of trajectory between
two such consecutive visits to $H$ is not greater than $\log\beta^-_{02}-\varepsilon$.

Note that it is enough to consider the case that $\xi_i\ne 1$ for infinitely many $i\ge 0$ since otherwise we would have $\overline\chi(p,\xi)<0$.
Further, we can freely assume that $p\ne0$ as otherwise we could replace $p$ by some iterate.
Note that
\[
	\beta'
		= \max\{f_0'(x),\lvert f_1'(x)\rvert,f_2'(x)\colon x\notin H\}<\beta^-_{02}.
\]
Thus, we can further assume that the orbit $\{f_{[\xi_0\ldots\,\xi_m]}(p)\}_{m\ge 0}$ hits the interval $H$ infinitely many times. Indeed, otherwise this orbit would be contained in the interval $(\de,1]$ in which the derivatives $f_0'$, $\lvert f_1'\rvert$, and $f_2'$ are bounded from above by $\beta'$ and thus the upper Lyapunov exponent $\overline\chi(p,\xi)$ would be bounded from above by $\log \beta'<\log\beta^-_{02}$.
Hence, without loss of generality, possibly replacing $p$ by some positive iterate, we can assume that $p \in H$ and $f_{\xi_0}(p)\notin H$.

For every $m\ge 0$ we write
$
	p_{m+1}\eqdef f_{[\xi_0\ldots\,\xi_m]}(p).
$	
Note that by our choices in~\eqref{e.H0} the only way of entering in $H$ is by coming from $H'$ after applying $f_1$ and the only way of entering and staying in $H'$ is by applying $f_0$.
\begin{figure}[h]
\vspace{0.6cm}
\begin{minipage}[t]{\linewidth}
\centering
\begin{overpic}[scale=.5,
  ]{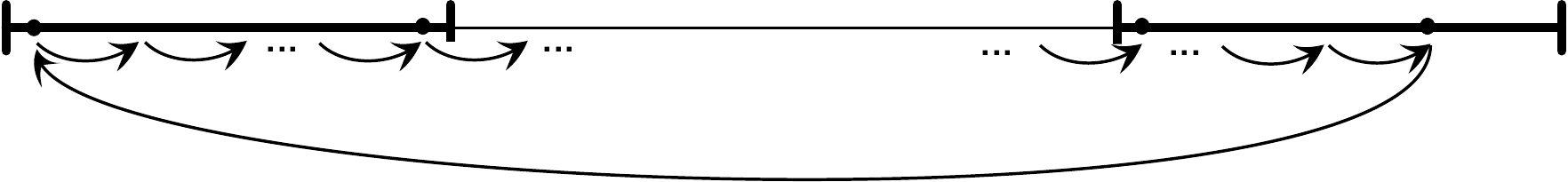}
    \put(-0.5,4){$0$}	
    \put(99,4){$1$}
    \put(1.4,14){$p_{r_k}$}
    \put(26,14){$p_{e_k}$}
    \put(71,14){$p_{i_k}$}
    \put(90,14){$p_{r_k-1}$}
    \put(13,17){$H$}
    \put(82,17){$H'$}
  \end{overpic}
\end{minipage}
\caption{Definition of the sequences $(r_k)_k$, $(e_k)_k$, $(i_k)_k$}
\label{f.seiiq}
\end{figure}
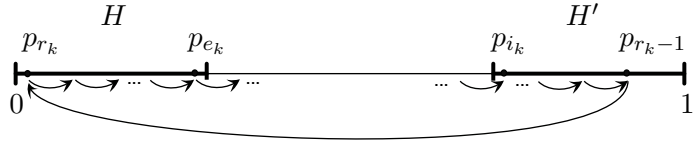

We define three increasing sequences $(r_k)_k$, $(e_k)_k$, and $(i_k)_k$ of positive integers as follows: $i_k< r_k \le e_k<i_{k+1}$,
\begin{equation}\label{e.rkek}
	p_j\in H
	\quad\text{ if and only if }\quad
	r_k\le j \le e_k \text{ for some }k,
\end{equation}
and
\begin{equation}\label{e.ik}
	p_j\in H'\quad\text{ for all }\quad
	i_k\le j\le r_k-1 \text{ for some }k,
\end{equation}
where $i_k$ is the smallest number with this property.
By the above observation $\xi_{r_k-1}=1$.
Condition~\eqref{e.H0} implies that $\xi_j\ne1$ for every index $j\in \{r_k,\ldots,e_k-1\}$ whenever $r_k<e_k$.
By~\eqref{e.H0} we also have $\xi_j=0$ for every index $j\in \{i_k-1,\dots,r_k-2\}$. In particular, writing $H'=[\delta',1]$, the latter implies
\begin{equation}\label{e.zeros}
	\delta'\le p_{i_k} < f_0(\delta').
\end{equation}
Moreover, we have $p_j\notin H$ for all $j\in \{e_k+1, i_{k+1}-1\}$, which implies \begin{equation}
\label{e.estimate1}
 	\log\big\lvert
	\big( f_{[\xi_{e_k+1}\,\ldots \,\xi_{i_{k+1}-1}]} \big)'(p_{e_{k}+1})\big\rvert
	< (\beta')^{i_{k+1}-e_k -1}.
\end{equation}
Let us denote by $N_k$ the number of iterates of the point $p_{i_k}$ staying in $H'$ before entering $H$:
\[
	N_k\eqdef r_k-i_k-1.
\]

\begin{claim}\label{c.alphaa}
	We have
	$\displaystyle p_{r_k}\ge \lambda_0^{N_k+1}\,\delta$.
\end{claim}

\begin{proof}
	By~\eqref{e.zeros} we have 
	\[
		f_0^{N_k}(p_{i_k})
		= p_{r_k-1}
		< f_0^{N_k}(f_0(\delta'))=f_0^{N_k+1}(\delta').
	\]	
	Since $\delta'=f_1^{-1}(\delta)$ and since $f_1$ is affine, we have $1- \delta'=\gamma^{-1}\delta$. Hence, as $f_0'\ge\lambda_0$ and $f_0(1)=1$, we have
	\[
		1- p_{r_k-1} > 1 -  f_0^{N_k+1}(\delta')
		\ge
		\lambda_0^{N_k+1}\,\gamma^{-1}\delta\,.
	\]
	Finally, we have
	\[
		p_{r_k}=f_1(p_{r_k-1})\ge
		\gamma\,\lambda_0^{N_k+1}\gamma^{-1}\delta
		=\lambda_0^{N_k+1}\,\delta,
	\]
	which proves the claim.
\end{proof}

By Claim~\ref{c.alphaa}, $e_k-r_k$ is bounded from above by $\widetilde M_k+1$, where $\widetilde M_k$ is defined by
\[
	(\beta_H)^{\widetilde M_k} \lambda_0^{N_k+1} \, \de
	= \de,
\]	
that is,
\begin{equation}\label{e.utt}
	\widetilde M_k+1\le M_k\eqdef
	\left\lfloor
	\frac{(N_k+1)\, \lvert \log \lambda_0\rvert }
		{\log \beta_H}\right\rfloor+2.
\end{equation}

Let us now estimate the finite-time Lyapunov exponent associated to the sequence $(\xi_{i_k}\dots \xi_{e_k})$.

\begin{claim}\label{c.square}
	We have $\displaystyle
	\frac{\log\big\lvert \big( f_{[\xi_{i_k}\ldots \,\xi_{e_k}]}\big)'(p_{i_k})
 	\big\rvert}{e_k-i_k+1} < \log\beta'$.
\end{claim}

\begin{proof}
First observe that if $e_k-i_k\le L$, with
\[
	\frac{\log\big\lvert \big( f_{[\xi_{i_k}\ldots \,\xi_{e_k}]}\big)'(p_{i_k})
 	\big\rvert}{e_k-i_k+1}
	\le\max_{\ell=1,\ldots,L} \frac{\log((\beta^+_{02})^\ell\,\gamma)}{\ell+1}
	\le \frac{\log((\beta^+_{02})^L\,\gamma)}{L+1}
\]
by our hypothesis~\eqref{e.derivatives2} the claim is automatically satisfied.

To prove the claim in the other case $e_k-i_k>L$, it is enough to assume that the number of iterations in the interval $H$ is the maximum possible (clearly this is the
case which bounds the derivative $(f_{[\xi_{i_k}\dots\, \xi_{e_k}]})'$ from above), that is, let us suppose that $e_k-r_k=M_k$. Then
\[
\frac{\log\big\lvert \big( f_{[\xi_{i_k}\ldots\, \xi_{e_k}]}\big)'(p_{i_k})
	\big\rvert}{e_k-i_k+1}
 \le \frac{M_k\log\beta^+_{02} + \log\gamma - N_k\,\lvert\log \,\lambda'\rvert}
		{M_k+N_k+1}.
\]
From $\log\gamma<0$ and~\eqref{e.utt} we conclude
\[\begin{split}		
&\frac{\log\big\lvert \big( f_{[\xi_{i_k}\ldots\, \xi_{e_k}]}\big)'(p_{i_k})
	\big\rvert}{e_k-i_k+1}\\
&\le \left[
		\frac{(N_k + 1) \, \lvert\log \lambda_0 \rvert}{\log \beta_H}\log \beta^+_{02}
			+2\log\beta^+_{02} - N_k\lvert\log \,\lambda'\rvert
	\right] \frac{1}{M_k+N_k+1}\\
&\le
	\frac{N_k}{M_k+N_k+1}
		\Big(
			\lvert\log \lambda_0 \rvert\,\frac{\log \beta^+_{02}}{\log \beta_H} 	
			- \lvert\log \,\lambda'\rvert
		\Big)		\\
&\phantom{\le}		
	+ \Big( \lvert\log \lambda_0 \rvert
		+ 2\log\beta_H
	   \Big)\frac{\log \beta^+_{02}}{\log \beta_H} \,
	   \frac{1}{M_k+N_k+1}
	   \\
&\le
	\left(
		\lvert\log \lambda_0 \rvert\,\frac{\log \beta^+_{02}}{\log \beta_H}
		- \lvert\log \,\lambda'\rvert
		+ \frac 2 3\log\beta^+_{02}
		\right)
	+ \lvert\log \lambda_0 \rvert \, \frac{\log \beta^+_{02}}{\log \beta_H} \,
		\frac{1}{M_k+N_k+1}.
\end{split}\]
Hence, noting that $M_k+N_k\ge L$ by~\eqref{e.derivatives} and~\eqref{edefL} we  conclude that
\[
\frac{\log\big\lvert \big( f_{[\xi_{i_k}\ldots\, \xi_{e_k}]}\big)'(p_{i_k})
	\big\rvert}{e_k-i_k+1}
< \frac 3 4 \log \beta'
	+ \frac 1 4 \log\beta' =  \log\beta'.
\]
We have proved the claim.
\end{proof}

We are now ready to get an upper bound for $\overline\chi(p,\xi)$. It is
enough to consider segment of orbits corresponding to exit times
$e_k$ and starting at the point $p_1$:
\[
\begin{split}
\frac{ \log\big\lvert \big( f_{[\xi_{1}\ldots \,\xi_{e_k}]} \big)'(p_{1})
	\big\rvert}{e_k}
=&  \sum_{j=0}^{k}  \left( \frac{e_j-i_j+1}{e_k} \right) \,
\frac{ \log\big\lvert \big( f_{[\xi_{i_j}\ldots\, \xi_{e_j}]}\big)'(p_{i_j})
	\big\rvert}{e_j-i_j+1}  \\
& + \left( \frac{i_{j+1}-e_j-1}{e_k} \right)\,
 \frac{
\log\big\lvert \big( f_{[\xi_{e_j+1}\ldots\, \xi_{i_{j+1}-1}]}\big)'(p_{e_{j+1}})
	\big\rvert}{i_{j+1}-e_j-1}.
\end{split}
\]
By equation \eqref{e.estimate1} and Claim~\ref{c.square} we get that
this derivative is bounded from above by $\log\beta'<\log\beta^-_{02}$. This completes  the proof of the proposition.
\end{proof}

\section{Thermodynamical formalism}

In the first part of this section we establish the existence of rich phase transitions using the gap in the Lyapunov spectrum (Theorem~\ref{t.hund}). In the second part we construct a maximal entropy measure with nonpositive central exponent.

\subsection{Co-existence of equilibrium states with positive entropy}\label{s:5}

Due to of the skew product structure and our hypotheses, the splitting in~\eqref{e.splitt} is dominated and for every Lyapunov regular point $R\in \La_F$ coincides with the Oseledec splitting provided by the multiplicative ergodic theorem. In particular, the \emph{Lyapunov exponent associated to the central direction $E^c$} at such a point  $R$ is well-defined and, in fact, is the Birkhoff average of the continuous function $R\mapsto\log\,\lVert dF|_{E^c_R}\rVert$
\[
    \chi_c(R) \eqdef
    \lim_{n\to\infty}\frac{1}{n}\log \,\lVert {dF^n|}_{E^c_R}\rVert
    = \lim_{n\to\infty}\frac 1 n
    	\sum_{k=0}^{n-1}\log \,\lVert {dF|}_{E^c_{F^k(R)}}\rVert.
\]
Observe that given a Lyapunov regular point $R=(r^s,r^u,r)\in\Lambda$ and a sequence $\xi=(\ldots\xi_{-1}.\xi_0\xi_1\ldots)\in\Sigma_3$ given by
$\xi=\varpi(r^s,r^u)$, we have
\[
    \chi_c(R)
    = \lim_{n\to\infty}\frac 1 n \log\,\lvert (f_{[\xi_0\ldots \,\xi_{n-1}]})'(r)\rvert.
\]
Finally note that the remaining exponents are associated to the stable and the unstable directions $E^s$ and $E^u$, respectively, and are uniformly bounded away from zero.

Let us first recall some general facts. We denote by $\cM(A)$ the set of $F$-invariant Borel probability measures supported on a set $A\subset \Lambda$ and by $\cM_{\rm e}(A)$ the subset of ergodic measures. For $\mu\in\cM(\Lambda)$ let
\[
    \chi_c(\mu)\eqdef \int\log\,\lVert dF|_{E^c}\rVert\,d\mu.
\]
Considering the \emph{spectrum of ergodic measures}, based on Proposition~\ref{p.beta}, the following result about the set of all possible central exponents is an immediate consequence.

\begin{coro}\label{c.spectrum}
    $\displaystyle
    	\big\{\chi_c(\mu)\colon\mu\in\cM_{\rm e}(\Lambda)\big\} \subset
	[\log\lambda_0,\log\widetilde\beta\,]\cup[\log\beta^-_{02},\log\beta^+_{02}].
    $
\end{coro}

\begin{rema}\label{r.spectruum}{\rm
	Note that there are two possibilities for an ergodic measure.
	If its support is contained in $\Lambda_{02}$ then its central Lyapunov exponent is contained in $[\log\beta^-_{02},\log\beta^+_{02}]$. Otherwise, any generic point is outside $\Lambda_{02}$ and then by Proposition~\ref{p.beta} has central Lyapunov exponent in $ [\log\lambda_0,\log\widetilde\beta\,]$.
}\end{rema}

Given a continuous potential $\varphi\colon\Lambda\to\RR$, an $F$-invariant Borel probability measure $\nu$ is called an \emph{equilibrium state} of
$\varphi$ with respect to $F|_{\Lambda}$ if
\[
    h_\nu(F)+\int\varphi\,d\nu
    = \max_{\mu\in\cM(\Lambda)}\Big( h_\mu(F)+\int\varphi\,d\mu\Big),
\]
where $h_\mu(F)$ denotes the measure theoretic entropy of $\mu$. Since the central direction is $1$-dimensional such maximizing measure indeed exists by~\cite[Theorem A]{DiaFis:11}. Note that we have the following \emph{variational principle}
\begin{equation}\label{varprinc}
    P_{F|\Lambda}(\varphi) =
    \max_{\mu\in\cM_{\rm e}(\Lambda)}\Big( h_\mu(F)+\int\varphi\,d\mu\Big),
\end{equation}
where $P_{F|\Lambda}(\varphi)$ is the \emph{topological pressure} of $\varphi$ with respect to $F|_\Lambda$ (see~\cite{Wal:81} for the definition and further properties that are used in the following). Denote by $h(F)=h(F|_{\Lambda})=P_{F|\Lambda}(0)$ the \emph{topological entropy} of $F|_\Lambda$. Note that an equilibrium state $\nu$ for the zero potential $\varphi=0$ is a measure of maximal entropy $h_\nu(F)=h(F)$.
It is immediate that the two-legged horseshoe $\Lambda_{02}\subset\Lambda$ defined in~\eqref{e.twoleg} satisfies
\begin{equation}\label{e.twolegent}
	h(F|_{\Lambda_{02}})=\log2.
\end{equation}
As the central direction does not contribute to the entropy, we also have
\begin{equation}\label{e.ipane}
	h(F|_{\Lambda})=\log3
\end{equation}
(compare arguments in~\cite[Section 3]{DiaFis:11} and~\cite{BuzFisSamVas:}).

Let us investigate the following one-parameter family $\varphi_t$ of continuous potentials defined by
\[
    \varphi_t\eqdef -t\log\,\lVert dF|_{E^c}\rVert,
    \quad t\in\RR,
\]
and will denote $P(t)\eqdef P(\varphi_t)$. Note that $t\mapsto P(t)$ is convex (and hence continuous and differentiable on a residual set). One says that $P$
exhibits a \emph{phase transition} at a characteristic parameter $t_c$ if it fails to be real analytic at $t_c$. We say that it has a \emph{first order phase transition} at $t_c$ if it fails to be differentiable at $t_c$.

\begin{rema}\label{r.convexx}{\rm
	Let us recall some basic facts. A number $\alpha\in \RR$ is said to be a \emph{sub-gradient}  at $t$ if $P(t+s)\ge P(t)+s\,\alpha$ for all $s\in \RR$.  Note that any equilibrium state $\mu_t$ of the potential $\varphi_t$ with respect to $F|_\Lambda$ provides a sub-gradient of $s\mapsto P(s)$ at $s=t$ given by  $-\chi_c(\mu_t)$. By definition
    \begin{equation}\label{e.entropyy}
        P(t)-t\chi_c(\mu_t) = h_{\mu_t}(F).
    \end{equation}
In particular, the entropy of $\mu_t$ is the intersection of the tangent line $s\mapsto P(t) -\chi_c(\mu_t)(s-t)$ with the $y$-axis.
Moreover, if $s\mapsto P(s)$ is differentiable at $s=t$ then
    \begin{equation}\label{e.deriva}
    \chi_c(\mu_t)=\alpha(t)\eqdef-P'(t) .
    \end{equation}
In particular, in this case, all equilibrium states  of $\varphi_t$ have the same exponent. In our case, non-differentiability is equivalent to the existence of a parameter $t$ and (at least) two equilibrium states for $\varphi_t$ with different central exponents.
}\end{rema}

A first order phase transition of $P$ is said to be \emph{rich} if there are two  associated equilibrium states with different central exponents and with positive entropy.

The next result establishes the existence of a rich phase transition.

\begin{figure}
\begin{minipage}[c]{\linewidth}
\centering
\begin{overpic}[scale=.4,
  ]{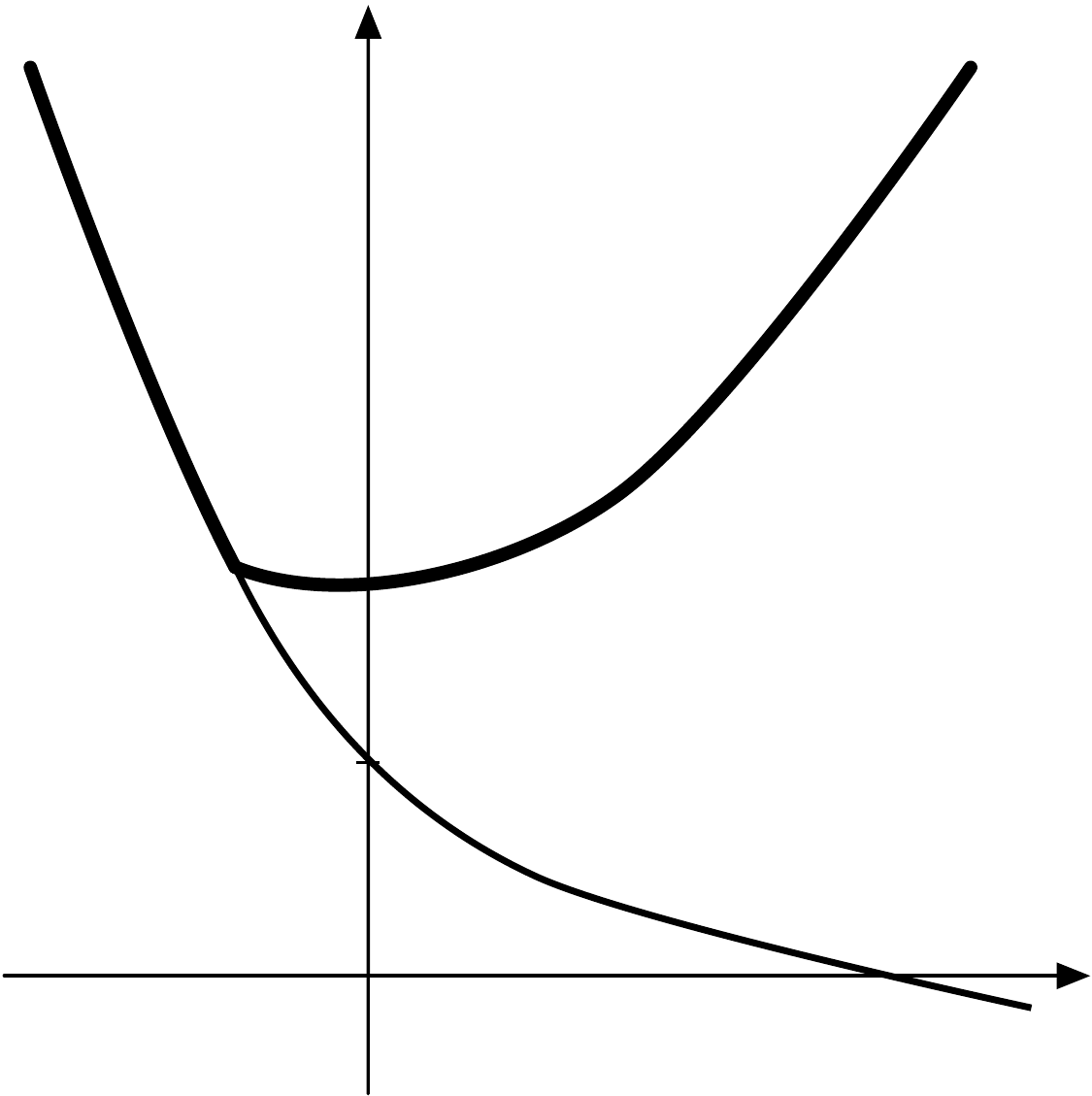}
      	\put(103,7.5){\small$t$}	
	\put(76,70){$P(t)$}	
	\put(59,19){$P_{F|\Lambda_{02}}(\varphi_t)$}	
	\put(35,41){\tiny$\log 3$}	
	\put(35,31){\tiny$\log 2$}	
\end{overpic}
\hspace{0.5cm}
\begin{overpic}[scale=.4,
  ]{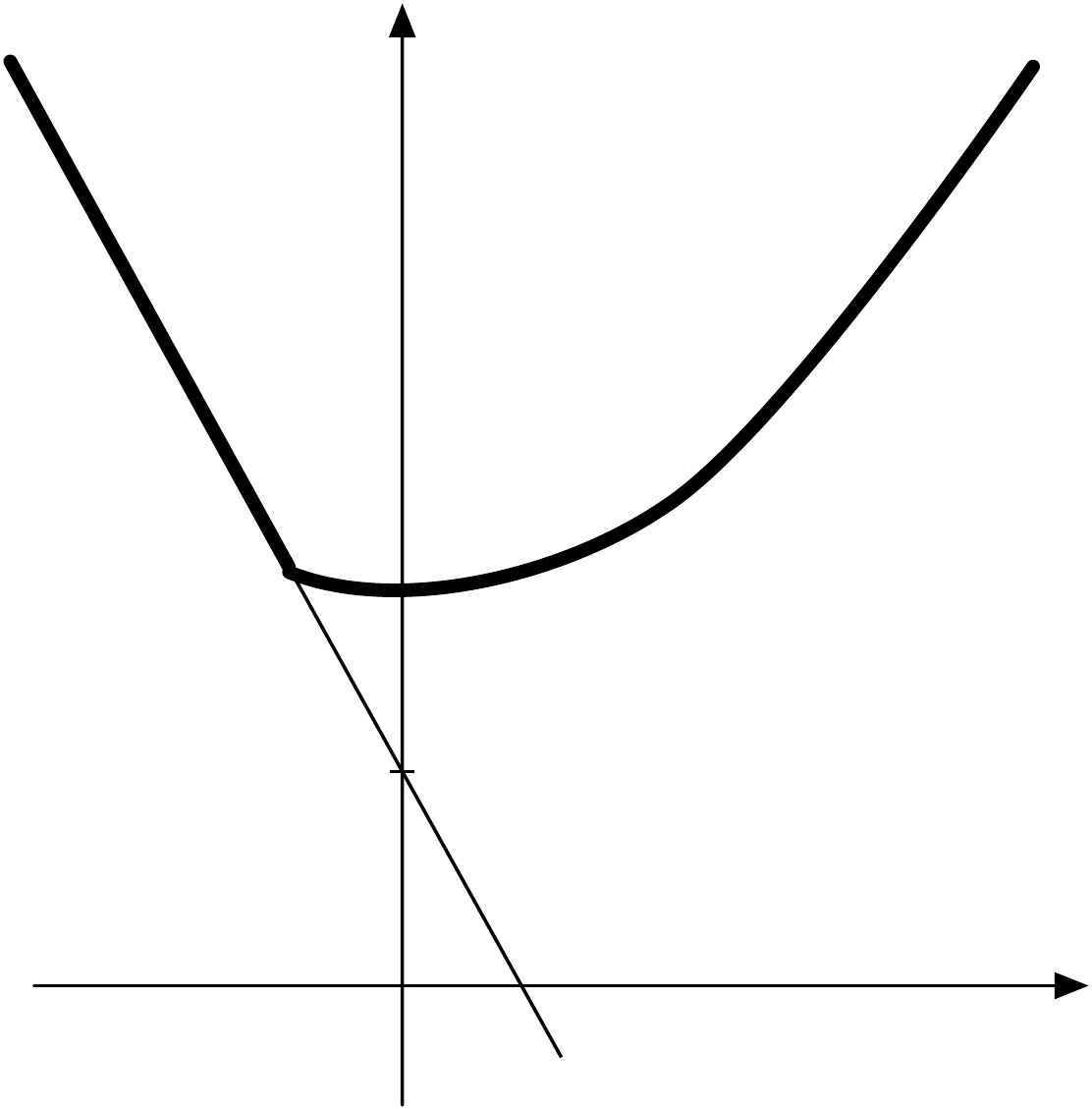}
      	\put(103,7.5){\small$t$}	
	\put(81,70){$P(t)$}	
	\put(46,16){$P_{F|\Lambda_{02}}(\varphi_t)$}	
	\put(38,41){\tiny$\log 3$}	
	\put(38,31){\tiny$\log 2$}	
\end{overpic}
\end{minipage}
\caption{The functions $t\mapsto P(t)$ and $t\mapsto P_{F|\Lambda_{02}}(\varphi_t)$ for $\beta_0\ne\beta_2$ (left) and $\beta_0=\beta_2$ (right).}\label{fig.map2}
\end{figure}
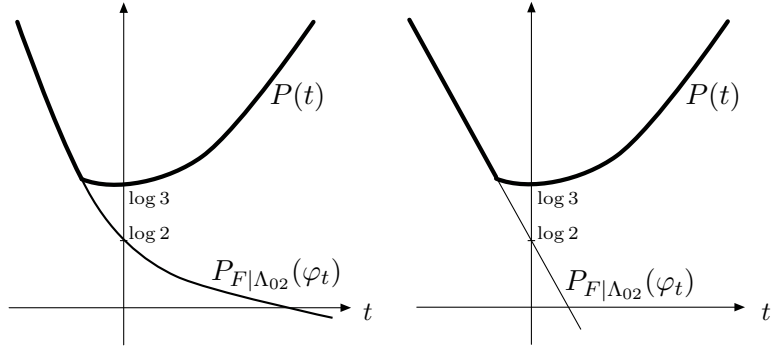

\begin{theo}[Rich phase transition]\label{t.hund}
    There are numbers $\beta_c^-\in[\beta_{02}^-,\beta_{02}^+]$
    and $\beta_c^+\in\big(\lambda_0,\widetilde\beta\,\big)$, $\beta_c^->\beta_c^+$, and a parameter $t_c<0$ such that $P$ is real analytic in $(-\infty,t_c)$ and not differentiable at $t_c$. Moreover it satisfies
    \[
        D^-P(t_c)=-\log\beta_c^-
        \quad\text{ and }\quad
        D^+P(t_c)=-\log\beta_c^+.
    \]
    Further, there exist equilibrium states $\mu^+$ and  $\mu^-$ for $\varphi_{t_c}$ with respect to $F|_{\Lambda}$ that both have positive entropy and central Lyapunov exponents $\log\beta_c^+$ and $\log\beta_c^-$, respectively.
\end{theo}

Indeed the number $t_c$ is the largest parameter where the pressure of the lateral horseshoe $F|_{\Lambda_{02}}$ and of the full system $F|_\Lambda$ coincide.

We need the following preliminary result.

\begin{lemm}\label{l.hund}
	We have
	\[
		\lim_{t\to\infty}\frac{P(t)}{t}=-\inf_{\nu\in\cM}\chi(\nu)
		=\lim_{t\to\infty}D^+P(t)
	\]
	and
	\[	
		\lim_{t\to-\infty}\frac{P(t)}{t}
		=-\sup_{\nu\in\cM}\chi(\nu)
		=\lim_{t\to-\infty}D^-P(t)
		.		
	\]
	If $t\mapsto P(t)$ is differentiable on $\RR$ then
	\[
		\{P'(t)\colon t\in\RR\}
		= \{-\chi(\nu)\colon\nu\in\cM(\Lambda)\}.
	\]
\end{lemm}

\begin{proof}
	Note that $0\le h_\nu(F)$ and the variational principle imply
		\[
		\frac{1}{t}\sup_{\nu\in\cM_{\rm e}}(-t\chi(\nu))
		\le \frac{P(t)}{t}
		\le\frac{h(F|_\Lambda)}{t}+\frac 1 t \sup_{\nu\in\cM_{\rm e}}(-t\chi(\nu))
	\]
	and hence the first equality. By convexity, for any $t>s$
	\[
		P(s)+ \sup_{\tau\ge s}D^+P(\tau)\cdot(t-s)\ge P(t)
		\ge P(s)+D^+P(s)\cdot(t-s)
	\]
	and hence $\lim_{t\to\infty}P(t)/t=\lim_{s\to\infty}D^+P(s)$.
	Similar arguments can be made for the other equalities.	
	
	If $P$ is differentiable, by convexity $\{P'(t)\colon t\in\RR\}$ is an interval $I$.
	For every $t$ any equilibrium state $\nu_t$ for $\varphi_t$ satisfies $-\chi(\nu_t)=P'(t)\in I$. Thus proving the inclusion ``$\subset$''.
	On the other hand, by the first part of the lemma, for every $\nu\in\cM_{\rm e}$, $-\chi(\nu)\in I$, proving ``$\supset$''.
\end{proof}

 \begin{proof}[{Proof of Theorem~\ref{t.hund}}]
We now prove that $P$ is not differentiable everywhere. Arguing by contradiction, suppose that $t\mapsto P(t)$ is differentiable at every point. By Lemma~\ref{l.hund} and our construction
\[
	\overline{\{P'(t)\colon t\in\RR\}}=[-\log\beta^+_{02},-\log\lambda_0].
\]
Thus, in this case, for every $\alpha\in (\lambda_0,\beta^+_{02})$ there would exist $t\in\RR$ with $P'(t)=-\log\alpha$. In particular, there is an equilibrium state $\nu$ of $\varphi_t$ with $\chi_c(\nu)=-\log\alpha$. Recall that any ergodic component of $\nu$ is also an equilibrium state of $\varphi_t$, and hence, with the above, has the same central exponent. But if $\alpha\in(\widetilde\beta,\beta^-_{02})$ this contradicts Corollary~\ref{c.spectrum}. This implies that $P$ is not differentiable at some point.

Any equilibrium state in the two-legged horseshoe $\Lambda_{02}\subset\Lambda$ has exponent within $[\log\beta^-_{02},\log\beta^+_{02}]$. Moreover, as $F|_{\Lambda_{02}}$ is a locally maximal uniformly hyperbolic set, the function $t\mapsto P_{F|\Lambda_{02}}(\varphi_t)$ is real analytic. It is strictly convex
 if and only if $\beta_2\ne\beta_0$. Otherwise, it is linear and equal to $\log2-t\log\beta_0$.
Further, inside the hyperbolic horseshoe $\Lambda_{02}$, for every $\alpha\in(\beta^-_{02},\beta^+_{02})$ there exists a unique equilibrium state $\nu_\alpha$ with respect to $F|_{\Lambda_{02}}$ with exponent $\chi_{\rm c}(\nu_\alpha)=\log\alpha$. Finally note that $\Lambda_{02}\subset\Lambda$ immediately implies $P_{F|\Lambda_{02}}(\varphi_t)\le P(t)$ for all $t$.

\begin{prop}
	There exists $t_c\in\RR$ such that
	\[
		t_c=\sup\{s\colon P(t)=P_{F|\Lambda_{02}}(\varphi_t)\text{ for all }
				t\le s\}
		<0		.
	\]
	In particular, $P$ is real analytic in $(-\infty,t_c)$.
	Moreover, we have
	\[
		D^-P(t_c) < D^+P(t_c).
	\]
\end{prop}

\begin{proof}
 	Note that $t\mapsto P_{F|\Lambda_{02}}(\varphi_t)$ is analytic and all of its derivatives are in the interval $[-\log\beta^+_{02},-\log\beta^-_{02}]$.	By Lemma~\ref{l.hund}, for $t$ small enough the equilibrium state $\mu_t$ for $\varphi_t$ has central Lyapunov exponent arbitrarily close to $\log\beta^+_{02}$ and hence not contained in the interval $[\log\lambda_0,\log\widetilde\beta\,]$. By Remark~\ref{r.spectruum} this state is thus supported in $\Lambda_{02}$. Therefore, $P_{F|\Lambda_{02}}(\varphi_t)=P(t)$ for every $t$ small enough.  Thus, $t_c>-\infty$.
Note that Lemma~\ref{l.hund} implies that $t_c<\infty$.	
To see that $t_c<0$, just note that
	\[
		P(0)=h(F|_\Lambda)=\log3>
		P_{F|\Lambda_{02}}(0)=h(F|_{\Lambda_{02}})=\log2.
	\]	
This proves the first part of the proposition.
	
	Arguing by contradiction, let us assume that $P$ is differentiable at $t_c$,
	\[
		D^-P(t_c) =D^+P(t_c)=P'(t_c)=P'_{F|\Lambda_{02}}(t_c).
	\]	
This implies that $D^+P(t_c)\le -\log\beta^-_{02}$. As for $t>t_c$ and close to $t_c$ we have $P(t)>P_{F|\Lambda_{02}}(\varphi_t)$ there exists an ergodic equilibrium state $\mu_t$ of $\varphi_t$ with respect to $F|_{\Lambda}$. By Remark~\ref{r.spectruum} it has generic points outside $\Lambda_{02}$. Thus $\chi(\mu_t)\le\log\widetilde\beta<\log\beta^-_{02}$ and hence $D^+P(t)\ge-\log\widetilde\beta>-\log\beta^-_{02}$.

Consider $\kappa\eqdef \inf_{t>t_c}D^+P(t)\ge-\log\widetilde\beta$ and observe that for $t>t_c$ and $s\ge t$ we have
	$
		P(s)\ge P(t)+\kappa (s-t)
	$	
and thus $P(s)\ge P(t_c)+\kappa(s-t_c)$. This implies $P'(t_c)\ge \kappa\ge-\log\widetilde\beta$, contradicting that $P'(t_c)\le -\log\beta^-_{02}$. Hence, $P$ is not differentiable at $t_c$.
\end{proof}

Let
\[
	\beta_c^-\eqdef \exp(- D^-P(t_c))
	\quad\text{ and }\quad
	\beta_c^+\eqdef \exp(- D^+P(t_c)).
\]	
With the above, $\beta_c^+\le\widetilde\beta$ and either $\beta_c^-=\beta_2=\beta_0$ or $\beta_c^-\in (\beta^-_{02},\beta^+_{02})$. Let us restrict to the latter case (the other one is similar and simpler).
There exists a unique (hence ergodic) equilibrium state $\mu^-$ of $\varphi_{t_c}$ with respect to $F|_{\Lambda_{02}}$ that has exponent $\beta_c^-$.
On the other hand, there exists an ergodic equilibrium state $\mu^+$ of $\varphi_{t_c}$ with respect to $F|_{\Lambda}$ that has exponent $\beta_c^+$.

\begin{prop}
	$\displaystyle 0<h_{\mu^-}(F)<h_{\mu^+}(F)$.
\end{prop}

\begin{proof}
	Recall~\eqref{e.entropyy} and observe that $h_{\mu^-}(F)$ is equal to the intersection of the tangent line to the pressure $P_{F|\Lambda_{02}}(\varphi_t)$ at $t=t_c$ with the $y$-axis. If $h_{\mu^-}(F)=0$ then, by strict convexity of $t\mapsto P_{F|\Lambda_{02}}(\varphi_t)$, the corresponding tangent line at $t<t_c$ would intersect the $y$-axis at some negative value providing negative entropy, which is impossible. This shows $h_{\mu^-}(F)>0$.
	
	Finally, to see that $h_{\mu^+}(F)$ is also positive, recall that $t_c<0$ and that the line $s\mapsto P(t_c)-\chi(\mu^+)(s-t_c)$ is above the tangent line to the pressure $P_{F|\Lambda_{02}}(\varphi_t)$ at $t=t_c$. Thus, its intersection with the $y$-axis is above $h_{\mu^-}(F)$ and hence is positive. 	
\end{proof}	
	
This proves the theorem.
\end{proof}

\subsection{Lifts of  Bernoulli measures
}\label{s.52}

In this section we study closer Lyapunov exponents of periodic pairs constructing lifts of Bernoulli measures and, in particular, a measure with maximal entropy. Note that, for a periodic sequence $(\xi_0\ldots\xi_{m-1})^\ZZ\in\Sigma_3$ and a fixed point $p= f_{[\xi_0\ldots\,\xi_{m-1}]}(p_{(\xi_0\ldots\,\xi_{m-1})^\ZZ})$, we have
\begin{equation}\label{e.lyapexpo}
    \chi(p,(\xi_0\ldots\xi_{m-1})^\ZZ)=
    \frac 1 m \log\, \big\lvert (f_{[\xi_0\ldots\,\xi_{m-1}]})'(p)\big\rvert.
\end{equation}

To the skew-product structure there is naturally associated a semiconjugation $\pi\colon\Lambda\to\Sigma_3$ such that $\sigma\circ\pi=\pi\circ F$. Recall also that $\varpi\colon \Gamma\to\Sigma_3$ conjugates the planar three-legs horseshoe and the shift, see Section~\ref{s:2}.
Given a $m$-periodic sequence $\xi=(\xi_0\ldots\xi_{m-1})^\ZZ$ (we do not assume that this period is minimal), consider the fiber $\pi^{-1}(\xi)$, that is, the intersection of $\Lambda$ with the line $\{\varpi^{-1}(\xi)\}\times [0,1]$ (it might be degenerated to a single point). Note that for the periodic sequence $\xi$ we have
\[
   \pi^{-1}(\xi)= \{\varpi^{-1}(\xi)\}\times I_{[\xi]},\quad
   I_{[\xi]} \eqdef \bigcap_{n\ge1} \big(f_{[\xi_0\ldots\,\xi_{m-1}]}\big)^n([0,1]).
\]

\begin{rema}{\rm
Note that for a periodic sequence $\xi$ the set $I_{[\xi]}$ is the maximal invariant set of the homeomorphism $g=f_{[ \xi_0\ldots\,\xi_{m-1} ]}$ on the interval $[0,1]$. If this set is a single point then it is a  topologically attracting fixed point for $g$. If this set is an interval, its endpoints must either be fixed points for $g$ (if $g$ is orientation-preserving) or for $g^2$ (if $g$ is orientation-reversing). Moreover, any endpoint of such an interval that is not an endpoint of the interval $[0,1]$ is (topologically) attracting. The point $1$ is an endpoint if and only if $\xi\ne0^\ZZ$, and in this case it is an attracting point for $g$. The point $0$ is an endpoint if and only if $\xi\in \{0,2\}^\ZZ$, and in this case it is a repelling point for $g$.
}\end{rema}

In view of the above remark we obtain following properties. Let $p_\xi^-$ and $p_\xi^+$ be the left and right endpoints of $I_{[\xi]}$ (they might be equal if $I_{[\xi]}$ is one point). If the symbol $1$ appears an even but nonzero number of times in the finite sequence $(\xi_0\ldots\xi_{m-1})$ then $(\varpi^{-1}(\xi),p_\xi^-)$ and $(\varpi^{-1}(\xi),p_\xi^+)$ are fixed points of $F^m$ and (topologically) attracting with respect to the central dynamics in the fibers. If the symbol $1$ appears an odd number of times, both those points are (topologically) attracting fixed points of $F^{2m}$. If the symbol $1$ does not appear, then $p_\xi^-=0$ and $(\varpi^{-1}(\xi),p_\xi^+)$ is a (topologically) attracting fixed point of $F^m$. Finally note that the central Lyapunov exponent at any periodic point which is topologically attracting in the central direction must be nonpositive.

\begin{prop}\label{p.maxent1}
There exists an ergodic $F$-invariant measure $\mu$ satisfying
\[
	h_\mu(F)=\log 3\quad\text{ and }\quad
	\chi(\mu) \leq 0.
\]	
\end{prop}

\begin{proof}
Denote by $S_m$ the set of all sequences of period $m$ that contain at least a symbol $1$. For $\xi\in S_m$ define
\[
	\mu_m \eqdef
	\frac 1 {\card S_m} \,\sum_{\xi\in S_m} 
		\frac 1 2 (\delta_{(\varpi^{-1}(\xi),p_\xi^-)} + \delta_{(\varpi^{-1}(\xi),p_\xi^+)})
\]
and let $\mu$ be a weak accumulation of $\mu_m$. Then $\mu$ is $F$-invariant and satisfies $\chi(\mu)\ge0$ by the above considerations.  To calculate the entropy of $\mu$ consider the projection $\nu_m=\pi^\ast\mu_m$ of $\mu_m$ to $\Sigma_3$. The measure $\nu_m$ is equally distributed on all sequences which are $m$-periodic with respect to $\sigma$ and which do not contain a symbol $1$. The proportion of these sequences in the set of all $m$-periodic ones is $1-(2/3)^m$, that is, it converges to $1$ as $m\to\infty$. Hence, the sequence $\nu_m$ converges weakly to the $(1/3, 1/3, 1/3)$-Bernoulli measure that has entropy $\log 3$. This Bernoulli measure is the projection of $\mu$. Hence, together with~\eqref{e.ipane} we can conclude that $h_\mu(F)=\log3$. 
\end{proof}

The following fact is a consequence of Proposition~\ref{p.maxent1} and Remark~\ref{r.convexx}.

\begin{coro}
	We have $\displaystyle D^+P(0)\le0$.
\end{coro}

An analogous construction can be done over $\{0,2\}^\ZZ$.

\begin{prop}
	There exist ergodic $F$-invariant measures $\mu_1$, $\mu_2$ with
	\[
		h_{\mu_1}(F)=h_{\mu_2}(F)=\log2\quad\text{ and }\quad
		\chi(\mu_1)\le0<\chi(\mu_2).
	\]
	Moreover, $\pi^\ast\mu_1=\pi^\ast\mu_2$ is the $(\frac12,0,\frac12)$-Bernoulli measure.
\end{prop}

\begin{proof}
	We will modify the construction in the proof of Proposition~\ref{p.maxent1}.  Consider the set of all $m$-periodic sequences $T_m\subset\{0,2\}^\ZZ$ and define
\[
	\mu_m^1\eqdef
	\frac{1}{\card T_m}\sum_{\xi\in T_m}\delta_{(\varpi^{-1}(\xi),p^+_\xi)}\quad
	\text{ and }\quad
	\mu_m^2\eqdef
	\frac{1}{\card T_m}\sum_{\xi\in T_m}\delta_{(\varpi^{-1}(\xi),0)}.
\]	
and let $\mu_1$ and $\mu_2$ be a weak accumulation of $\mu_m^1$ and $\mu_m^2$, respectively. By construction, $\chi(\mu_1)\le0$ follows as in the proof above. Further, by construction $\mu_2$ is supported on $\Lambda_{02}$, projects to the $(\frac12,\frac12)$-Bernoulli measure in $\{0,2\}^\ZZ$ and hence $h_{\mu_2}(F)=\log2$ and $\chi(\mu_2)=\frac12\log(\beta_0\beta_2)>0$.
As $\mu_1$ and $\mu_2$ have identical projections, they carry the same entropy.
\end{proof}

The above construction can be performed for more general Bernoulli measures. In particular, we obtain the following result.

\begin{prop}
	Any Bernoulli measure on $\Sigma_3$ has a lift to an $F$-invariant measure  with nonpositive central Lyapunov exponent.  
	
	Given a $(p_0,p_1,p_2)$-Bernoulli measure $\nu$ on $\Sigma_3$, then any of its lift to an $F$-invariant measure $\mu$ has central Lyapunov exponent bounded by
\[
	\chi(\mu)\le p_1\log\gamma+(p_0+p_2)\log\beta_{02}^+.
\]	
Moreover, if this bound is negative, then for $\nu$-almost every $\xi$ the fiber $I_{[\xi]}$ is trivial and hence the lifted measure is unique.
\end{prop}	

\begin{coro}\label{c.entuni}
	If $\gamma\,(\beta_{02}^+)^{2}<1$, then the measure of maximal entropy is unique and has negative central Lyapunov exponent.
\end{coro}

\section{More general potentials}\label{s:6}

Let mention how the observations in Sections~\ref{s:4} and~\ref{s:5} can be extended to more general potentials than the one $\varphi=-\log\,\lVert dF|_{E^c}\rVert$. Given a continuous function $\phi\colon \Lambda\to\RR$ and  a point $R\in\Lambda$, let us consider the \emph{(forward) Birkhoff average} of $\phi$ at $R$ defined by
    \[
        \chi_\phi(R)\eqdef \lim_{n\to\infty}\frac1n\sum_{n=0}^{n-1}\phi\big(F^n(R)\big),
    \]
whenever this limit exists. Indeed, having a gap in the spectrum of such averages  is open in the space of continuous potentials with the supremum topology.

Note that condition (F012) gives constraints to the variation of the potential $\varphi$ on the lateral horseshoe $\Lambda_{02}$ and to the value of $\varphi(P_0)$.
If we assume that a condition similar to (F012) is satisfied for the potential $\phi$ then as in Sections~\ref{s:4}  and~\ref{s:5} we get a gap in the spectrum of values $\chi_\phi$.  Roughly speaking, we need to require that the range $[\alpha^-_{02},\alpha^+_{02}]$ of the potential $\phi|_{\Lambda_{02}}$ should be sufficiently small on $\Lambda_{02}$ and that $\sup\{\phi(R)\colon R=(\theta,1)\in\Lambda\}$ is sufficiently smaller than $\alpha^-_{02}$.

Using these properties,
we can show that there exists numbers $\widetilde\alpha<\alpha^-_{02}\le\alpha^+_{02}$ such that $\chi_\phi(R)\in[\alpha^-_{02},\alpha^+_{02}]$ for every $R\in\Lambda_{02}$ and $\chi_\phi(R)\le\widetilde\alpha$ whenever $R\in\Lambda\setminus\Lambda_{02}$.
We refrain ourselves from stating the precise
conditions for the potentials and giving
all details of the proof of this fact.
Then a statement analogous to Corollary~\ref{c.spectrum} can be obtained saying that for every ergodic measure $\mu\in\cM_{\rm e}(\Lambda)$ and the average
\[
    \chi_\phi(\mu)\eqdef \int\phi\,d\mu
\]
we have that
\[
    \{\chi_\phi(\mu)\colon\mu\in\cM_{\rm e}(\Lambda)\}\subset
    [\min\phi,\widetilde\alpha\,]\cup[\alpha^-_{02},\alpha^+_{02}].
\]
If, moreover, $\phi|_{\Lambda_{02}}$ is H\"older continuous, then it is an immediate consequence that there exists a parameter $t_c\in\RR$ such that $t\mapsto P(t\phi)$ is, in fact, real analytic on $(-\infty,t_c)$.

Finally,
let us formulate a general sufficient condition for a phase transition. It is an immediate consequence of the fact that every equilibrium state is a subgradient of the pressure function and of the gap in the range of averages with respect to ergodic measures  supported on the lateral horseshoe $\Lambda_{02}$ and on $\Lambda\setminus\Lambda_{02}$, respectively.

\begin{prop}
    Given a continuous potential $\phi\colon \Lambda\to \RR$. If
    \[
        \inf_{R\in\Lambda_{02}}\phi(R)>
        \sup_{\mu\in\cM_{\rm e}(\Lambda)\setminus\cM_{\rm e}(\Lambda_{02})}\chi_\phi(\mu)
    \]
    then there exists $t_c\in \RR$ such that $t\mapsto P(t\phi)$ is continuous in $\RR$ and not differentiable in $t_c$.
\end{prop}

\section{Transitivity of $\Lambda$}\label{s:7}

In this section we prove Proposition~\ref{p.transitive} claiming that $\Lambda$ is a homoclinic class and hence is topologically transitive.

We start by describing the invariant manifolds of the saddles $P_i$ and $Q_j$ of the diffeomorphism $F$, see~\eqref{e.defPQ}. Recall that by hypothesis the planar horseshoe map $\Phi$ is affine, thus  we have
\[
    [0,1]\times \{\theta_i^u\}=W^\st_\loc (\theta_i,\Phi)\quad
    \text{ and }\quad
    \{\theta_i^s\}\times [0,1]=W^\ut_\loc (\theta_i,\Phi).
\]
The definition of $F$ in~\eqref{e.defF}  implies that
\begin{equation}\label{e.manifolds}
\begin{split}
    [0,1]\times \{\theta_i^u\} \times (0,1]\subset W^\st (P_i,F)
    ,\\
    \{\theta_i^s\} \times [0,1] \times \{p_i\} \subset W^\ut  (P_i,F)
    ,\\
    [0,1]\times \{\theta_i^u\} \times \{0\} \subset W^\st (Q_i,F)
    ,\\
    \{\theta_i^s\}\times [0,1] \times  [0,p_i) \subset W^\ut (Q_i,F).
\end{split}
\end{equation}
In what follows we write
\[
    W^\st_\loc (Q_i,F)= [0,1] \times \{(\theta_i^s,q_i)\}
    \quad\mbox{and} \quad
    W^\ut_\loc (P_i,F)= \{\theta_i^s\} \times [0,1]\times \{p_i\} .
\]	

\begin{rema}\label{r.heterodimensional}{\rm
The definition of $F$ and~\eqref{e.manifolds} immediately imply that the point $P_0$ and the saddles $Q_0$, $Q_2$ are involved in a heterodimensional cycle, that is, the stable manifold of $P_0$ meets the unstable one of $Q_i$ and the unstable manifold of $P_0$ meets the stable one of $Q_i$.
}\end{rema}

To prove Proposition~\ref{p.transitive}, we follow closely the arguments  in~\cite{DiaGel:}, so we only sketch the main ideas and explain the differences.

\begin{proof}[Proof of Proposition~\ref{p.transitive}]
It follows from the construction that the set $\Lambda$ splits into the following three sets:
\begin{itemize}
	\item[1)] the lateral horseshoe $\Lambda_{02}$,
	\item[2)]
the \emph{cycle set} $\displaystyle W^u(P_0,F)\cap W^s(\Lambda_{02},F)$ corresponding to ``heterodimensional cycles'', and
	\item[3)]
the \emph{inner points} $X=(x^s,x^u,x)\in\Lambda$ for which $F^i(X)=(x_i^s,x_i^u,x_i)$ satisfies $x_i\in(0,1)$ for infinitely $i\le0$.
\end{itemize}

The main  ingredient of the proof is the notion of expanding itineraries.
In our case, the existence of such  itineraries is guaranteed by condition (F01).

\begin{rema}[Expanding itineraries]
\label{r.expanding}
{\rm
There are $\kappa>1$ and small $b>0$ close to $0$ such that for any interval $J\subset [f_0^{-2}(b),b]$ there is a number $n(J)$ (uniformly bounded)
such that for every $x\in J$
\[
	\big(f_1\circ f_0^{n(J)}\big)(x)\in  (0,b]
	\quad\text{ and }\quad
	\lvert (f_1\circ f_0^{n(J)})'(x)\rvert \ge \kappa
\]

Let us see how this property follows from our assumptions. For simplicity let us assume linearity of $f_0$ close to $0$ and $1$. Note that there are
arbitrarily small $t>0$ and large $n$ such that $f_0^n([\beta_0^{-1}t,t])=[1-t,1-t\,\lambda_0 ]$, recalling the definitions of $\beta_0$ and $\lambda_0$ in (F0).
Take $b=t$.
Using monotonicity of $f_0'$ and the fact that $J\subset [\beta_0^{-2}\, t,t]$
we have that, for all $x\in J$,
$$
	(f_0^n)'(x) \ge \frac{\lambda_0^2 \, (1-\lambda_0)}{1-\beta_0^{-1}}\,.
$$
We let $n(J)=n$ if $f_0^n(J) \subset [1-t,1]$ and $n(J)=n+1$ otherwise.
This choice implies that for all $x\in J$ we have
$$
	(f_0^{n(J)})'(x) \ge \frac{\lambda_0^3 \, (1-\lambda_0)}{1-\beta_0^{-1}}\,.
$$
Thus, by (F01),
$$
	\lvert (f_1 \circ f_0^{n(J)})'(x)\rvert \ge
\gamma
\, \left(
	\frac{\lambda_0^3 \, (1-\lambda_0)}{1-\beta_0^{-1}} \right) >\kappa>1.
$$
Finally, by construction $f_0^{n(J)} (J) \subset [1-t,1-\lambda_0^2\,t]$ and
thus  $f_1 \circ f_0^{n(J)} (J) \subset (0, \gamma\,t] \subset (0,t]$.
}\end{rema}

Let us now fix $b\in(0,1)$ close to $0$ satisfying Remark~\ref{r.expanding}.
The following claim corresponds to~\cite[Lemma 3.8]{DiaGel:}.

\begin{step}
\label{s.expanding}
	Given any (non-trivial) closed interval $J\subset [f_0^{-2}(b),b]$ there is a finite sequence $\xi(J)=(\xi_0\dots \xi_m)$, $\xi_i\in\{0,1\}$, such that
\begin{itemize}
	\item
		the map $f_{[\xi(J)]}=f_{\xi_m}\circ\ldots\circ f_{\xi_0}$ has a unique expanding fixed point $q^\ast_J\in J$,
	\item
		the unstable manifold $W^u(q_J^\ast,f_{[\xi(J)]})$ contains $[f_0^{-2}(b),b]$.
\end{itemize}		
The sequence $\xi(J)$ is called the \emph{expanding sequence} of $J$.
\end{step}

Indeed, one has that $\xi(J)=(0^{n(J)}\,1\,0^{m(J)})$, where $n(J)$ is defined as in Remark~\ref{r.expanding} and $m(J)$ is the first positive number such that $(f_0^{m(J)} \circ f_1 \circ f_0^{n(J)})(J) \cap [f_0^{-1}(b),b)]$ is non-empty. Remark~\ref{r.expanding} implies that $m(J)\ge 0$ and the map $f_0^{m(J)} \circ f_1 \circ f_0^{n(J)}$ restricted to $J$ is uniformly expanding. Step~\ref{s.expanding} follows by concatenating several returns of $J$ as above. At some step the return of $J$ will cover $J$ in an expanding way.

Consider the periodic point $Q^\ast$ of $F$ associated to the periodic sequence $(\xi_0\dots\xi_m)^\ZZ$ and to the central coordinate
$q_J^\ast$. The second item in Step~\ref{s.expanding} implies the following
(for details see  \cite[Lemma~4.8]{DiaGel:} and~\cite[Remark 4.6]{DiaGel:}).

\begin{step} \label{s.qunstable}
	The unstable manifold $W^u(Q^\ast,F)$ transversely intersects the s-disk  $[0,1]\times \{ (x^u,x)\}$ for any $x\in (0,1)$.
	The stable manifold $W^s(Q^\ast,F)$ transversely intersects any vertical disk of the form $\{x^s\}\times[0,1]\times J$, where $J\subset(0,1)$ is a fundamental domain of $f_0$.
\end{step}

Note that the proof of the above step involves the dynamics of $F$ in  the cubes $\CC_0$ and $\CC_1$ only. Indeed, in what follows all the arguments only involve
iterates in $\CC_0$ and $\CC_1$ and the obtained points have orbits contained
in $\CC$.

Since by construction $W^s(Q^*,F)$ transversely intersect $W^u(Q_i,F)$,
for $i=0,2$, we have the following.

\begin{step} \label{s.sclosures}
	The stable manifolds $W^{s}(Q_0,F)$ and $W^{s}(Q_2,F)$ are contained in the closure of $W^{s} (Q^\ast,F)$. More precisely, for any segment
$$
	[0,1]\times\{(a^u,0)\} \subset W^{s}(Q_0,F),
$$
there is a sequence of horizontal segments,
$$
	\Delta_n^s = [0,1]\times \{(a_n^u, a_n)\} \subset W^{s} (Q^\ast,F)
$$
such that $a_n^u\to a^u$ and $a_n\to 0^+$.
A corresponding statement holds for $Q_2$.
\end{step}

Step~\ref{s.qunstable} and the cycle configuration  in Remark~\ref{r.heterodimensional} immediately imply the following
fact  for the unstable manifolds (somewhat similar to Step~\ref{s.sclosures}).

\begin{step} \label{s.uclosures}
	The strong unstable manifolds $W^{uu}(Q_0,F)$ and $W^{uu}(Q_2,F)$ are contained in the closure of $W^{u} (Q^*,F)$. More precisely, for any segment
$$
	\{a^s\} \times [0,1]\times \{0\} \subset W^{uu}(Q_0,F), \quad
	 a^s\in[0,1],
$$
there is $c>0$ such that for each $n\ge1$ there are a sequence of numbers $a_{n,k}^s$ with $a_{n,k}^s\to a^s$ as $n\to\infty$ and a sequence of intervals $J_{n,k}$ with $\bigcup_kJ_{n,k}\supset(0,c]$ such that the sequence of vertical rectangles $\Delta_n^u$ satisfy
$$
	\Delta_n^u =
	\bigcup_k \,\{a^s_{n,k}\} \times [0,1]\times J_{n,k}
	\subset	W^{u} (Q^*,F).
$$
A corresponding statement holds for $Q_2$.
\end{step}

As the lateral horseshoe  in 1) is contained in the closure of $W^{uu}(Q_0,F)\pitchfork W^s(Q_0,F)$, Steps~\ref{s.sclosures} and~\ref{s.uclosures} imply that $\Lambda_{02}$ is contained in $H(Q^\ast,F)$. This proves the first part of the proposition.

The fact that the cycle points in $\Lambda$
(points satisfying 2) above)
are contained in $H(Q^\ast,F)$ follows arguing exactly as in the previous case considering $\Lambda_{02}$ observing that
\[
	W^u(P_0,F)\subset \overline{W^u(Q^\ast,F)}\quad\text{ and }\quad
	W^s(\Lambda_{02},F)\subset W^s(Q^\ast,F)
\]
using corresponding versions of Steps~\ref{s.sclosures} and~\ref{s.uclosures} where $Q_0$ is replaced by any point in $\Lambda_{02}$.

It remains to study the inner points of $\Lambda$ in 3). Similar to~\cite[Proposition 4.11]{DiaGel:} we have the following.

\begin{step}\label{s.whatever}
	Let $X=(x^s,x^u,x) \in \Lambda$ be an inner point.  Given any small $\delta>0$ the stable segment centered at $X$,
$$
	\Delta^s_\delta(X)\eqdef[x^s-\delta, x^s+\delta]\times\{(x^u,x)\}
$$
transversely intersects $W^u(Q^\ast,F)$. Given a point
$$
	X(\delta)\eqdef
	(x^s(\delta), x^u,x)\in \Delta_\delta^s(X)\pitchfork W^u(Q^\ast,F)
$$
then for every small $\varepsilon>0$ the disk
$$
	\Delta^{cu}_\varepsilon(X(\delta))
 	\eqdef
	\{x^s(\delta)\}\times [x^u-\varepsilon, x^u+\varepsilon]\times
		[x-\varepsilon,x+\varepsilon]
	\subset W^u(Q^\ast,F).
$$
intersects $W^s(Q^\ast,F)$ transversely.
\end{step}

The above fact immediately implies that the rectangle
$\Delta^{cu}_\varepsilon(X(\delta))$ contains a transverse homoclinic point of $Q^\ast$. As $\delta$ and $\varepsilon$ can be chosen arbitrarily small, this
transverse homoclinic point can be taken arbitrarily close to $X$. Thus $X$ is in the class of $Q^\ast$. This proves the proposition.

\begin{proof}[Sketch of the proof of Step~\ref{s.whatever}]
The proof only involves the dynamics in the cubes $\CC_0$ and $\CC_1$. By expansion of $F^{-1}$ in the $x^s$-coordinate, since $X$ is an inner point, after finitely many iterations by $F^{-1}$ the set $F^{-i}(\Delta^s_\delta(X))$ contains a disk of the form $[0,1]\times \{(y^u,y)\}$, for some $y^u\in [0,1]$ and $y\in(0,1)$.
By Step~\ref{s.qunstable} such a disk intersects $W^\ast(Q,F)$, proving our  first assertion and providing $X(\delta)$.

The second assertion follows using the expanding itineraries in Step~\ref{s.expanding}. After positive  iterations of the disk $\Delta^{cu}_\varepsilon(X(\delta))$ by $F$ one gets a ``big disk'' of the form $[0,1]\times \{y^s\} \times(a,a')$, where $a,a'\in (0,1)$ are both close to $1$.
Using the heterodimensional cycle involving $P_0$ and $Q_0$, after further forward
iterates one gets a disk $[0,1]\times \{z^s\} \times (c,c')$, where $c,c' \in (f_0^{-2} (b),b)$. Thus we can consider the interval $J=(c,c')$ and apply Step~\ref{s.expanding}. In this way, a further forward iterate of $\Delta^{cu}_\varepsilon(X(\delta))$ contains a vertical disk of the form $[0,1]\times \{w^s\} \times J'$, where $J'$ contains a fundamental domain of $f_0$ in $[f_0^{-2}(b),b]$.
By
the second part of Step~\ref{s.qunstable} this vertical disk intersects $W^s(Q^\ast,F)$. This finishes the sketch of the proof.
\end{proof}

This completes the sketch of proof of the proposition
\end{proof}

\bibliographystyle{amsplain}

\end{document}